\documentclass[11pt,a4paper]{amsart}
\usepackage{amsmath,amssymb,amsthm,amscd}

\title{Locally well generated homotopy categories of complexes}
\date{May 14, 2010}

\author{Jan \v S\v tov\'\i\v cek}
\address{
Charles University in Prague \\
Faculty of Mathematics and Physics \\
Department of Algebra \\
Sokolovska 83, 186 75 Praha \\
Czech Republic
}
\email{stovicek@karlin.mff.cuni.cz}


\thanks{The author was supported by the Research Council of Norway through the Storforsk-project ``Homological and geometric methods in algebra'' and also by the grant GAUK 301-10/252216.}

\renewcommand{\iff}{if and only if }
\newcommand{\st}{such that }
\newcommand{\N}{\mathbb{N}}
\newcommand{\Z}{\mathbb{Z}}

\DeclareMathOperator{\Ext}{Ext}

\DeclareMathOperator{\Ker}{Ker}
\DeclareMathOperator{\Img}{Im}

\newcommand{\spanned}[1]{\langle #1 \rangle}
\newcommand{\cpx}[1]{{#1}}
\newcommand{\A}{\mathcal{A}}
\newcommand{\B}{\mathcal{B}}
\newcommand{\C}{\mathcal{C}}
\newcommand{\D}{\mathcal{D}}

\newcommand{\clL}{\mathcal{L}}
\newcommand{\clS}{\mathcal{S}}
\newcommand{\T}{\mathcal{T}}

\newcommand{\Modu}[1]{\hbox{\rm Mod-}{#1}}
\newcommand{\ModR}{\hbox{\rm Mod-}R}
\newcommand{\ModA}{\hbox{\rm Mod-}\!\A}

\newcommand{\FlatR}{\hbox{\rm Flat-}R}
\newcommand{\FlatA}{\hbox{\rm Flat-}\!\A}
\newcommand{\ProjR}{\hbox{\rm Proj-}R}
\newcommand{\ProjA}{\hbox{\rm Proj-}\!\A}
\newcommand{\fpB}{\hbox{\rm fp}(\B)}

\newcommand{\Ab}{\mathrm{Ab}}

\newcommand{\Add}{\mathrm{Add}}

\newcommand{\Der}{\mathbf{D}}
\newcommand{\Htp}{\mathbf{K}}
\newcommand{\Comp}{\mathbf{C}}
\newcommand{\Loc}{\mathrm{Loc}\,}

\theoremstyle{plain}
\newtheorem{thm}{Theorem}[section]
\newtheorem{prop}[thm]{Proposition}
\newtheorem{lem}[thm]{Lemma}
\newtheorem{cor}[thm]{Corollary}

\theoremstyle{definition}
\newtheorem{defn}[thm]{Definition}
\theoremstyle{remark}
\newtheorem*{rem}{Remark}

\newtheorem{expl}[thm]{Example}

\begin{document}

\begin{abstract}
We show that the homotopy category of complexes $\mathbf{K}(\mathcal{B})$ over any
finitely accessible additive category $\mathcal{B}$ is locally well
generated. That is, any localizing subcategory $\mathcal{L}$ in $\mathbf{K}(\mathcal{B})$
which is generated by a set is well generated in the sense of
Neeman. We also show that $\mathbf{K}(\mathcal{B})$ itself being well generated is
equivalent to $\mathcal{B}$ being pure semisimple, a concept which naturally
generalizes right pure semisimplicity of a ring $R$ for $\mathcal{B}= \textrm{Mod-}R$.
\end{abstract}

\maketitle

\section*{Introduction}

The main motivation for this paper is to study when the homotopy category of
complexes $\Htp(\B)$ over an additive category $\B$ is compactly
generated or, more generally, well generated.

\smallskip

In the last few decades, the theory of compactly generated
triangulated categories has become an important tool unifying concepts
from various fields of mathematics. Standard examples are the
unbounded derived category of a ring or the stable homotopy category
of spectra. The key property of such a category $\T$ is the Brown
Representability Theorem, cf.~\cite{N3,Kr2}, originally due to Brown~\cite{Br}:

\begin{quote}
Any contravariant cohomological functor $F: \T \to \Ab$ which sends
coproducts to products is representable.
\end{quote}

This theorem is an important tool and has been used in several
places. We mention Neeman's proof of the Grothendieck Duality
Theorem~\cite{N3}, Krause's work on the Telescope
Conjecture~\cite{Kr3,Kr4}, or Keller's representation theorem for
algebraic compactly generated triangulated categories~\cite{Kel}.

Recently, there has been a growing interest in giving criteria for
certain homotopy categories $\Htp(\B)$ to be compactly generated,
\cite{HJ,Jo1,Kr5,N2}. Here, $\B$ typically was a suitable subcategory
of a module category. The main reason for studying such homotopy categories
were results concerning the Grothendieck Duality
Theorem~\cite{IK,N2} and relative homological
algebra~\cite{Jo2}. There is, however, a conceptual
reason, too. Namely, every algebraic triangulated category is triangle
equivalent to a full subcategory of some homotopy category,
\cite[\S 7.5]{Kr2}.

It turned out when studying the homotopy category of complexes of projective
modules over a ring $R$ in~\cite{N2} that it is useful to consider well generated
triangulated categories in this context. More precisely,
$\Htp(\ProjR)$ is always well generated, but may not be compactly
generated. Well generated categories have been defined by
Neeman~\cite{N} in a natural attempt to extend results such as the Brown
Representability from compactly generated triangulated categories to a
wider class of triangulated categories.

\smallskip

Although one has already known for some time that there exist rather natural triangulated
categories, such as the homotopy category of complexes of abelian
groups, which are not even well generated, one
has typically viewed those as rare and exceptional cases.

We will give some arguments to show that this interpretation is not
very accurate. First, the categories $\Htp(\ModR)$ for
a ring $R$ are rarely well generated. It happens \iff $R$ is right
pure semisimple, which establishes the converse of~\cite[\S 4 (3),
p. 17]{HJ}. Moreover, we generalize this result to the homotopy
categories $\Htp(\B)$ with $\B$ additive finitely accessible. This
way, we obtain a fairly complete answer regarding when
$\Htp(\FlatR)$ is compactly or well generated, see~\cite[Question
4.2]{HJ}.

We also give a partial remedy for the typical failure of $\Htp(\B)$ to be
well generated. Roughly speaking, the main problem with $\Htp(\B)$, where
$\B$ is finitely accessible, is that it
may not have any set of generators at all. But if we take a
localizing subcategory $\clL$ generated by any set of objects, it will
automatically be well generated. We will call a triangulated category
with this property locally well generated.

We will also give basic properties of locally well generated
categories and see that some of the usual results regarding localization
hold in the new setting. For example, any localizing
subcategory generated by a set of objects is realized as the kernel
of a localization endofunctor. This version of a Bousfield
localization theorem generalizes~\cite[\S 7.2]{Kr}
and~\cite[5.7]{AJS}. However, one has to be more careful. The Brown
Representability theorem as stated above does not work for locally
well generated categories in general, and there are localizing
subcategories which are not associated to any localization
endofunctor. We illustrate this in
Example~\ref{expl:brow_rep_failure}.

\subsection*{Acknowledgements}
The author would like to thank Henning Krause for several helpful
discussions and suggestions, as well as for his hospitality during the
author's visits in Paderborn.

\section{Preliminaries}
\label{sec:prelim}

Let $\T$ be a triangulated category. A triangulated subcategory $\clS
\subseteq \T$ is called \emph{thick} if, whenever $X \amalg Y \in
\clS$, then also $X \in \clS$. From now on, we will assume that $\T$
has arbitrary (set-indexed) coproducts. A full triangulated
subcategory $\clL \subseteq \T$ is called \emph{localizing} if it is
closed under forming coproducts. Note that by~\cite[1.6.8]{N}, $\T$
has splitting idempotents and any localizing subcategory $\clL
\subseteq \T$ is thick.

If $\clS$ is any class of objects of $\T$, we denote by $\Loc\clS$ the
smallest localizing subcategory of $\T$ which contains $\clS$. In other
words, $\Loc\clS$ is the closure of $\clS$ under shifts, coproducts and
triangle completions.

Given $\T$ and a localizing subcategory $\clL \subseteq \T$, one can
construct the so-called \emph{Verdier quotient} $\T/\clL$ by formally
inverting in $\T$ all morphisms in the class $\Sigma(\clL)$ defined as
$$ \Sigma(\clL) = \{f \mid \exists \textrm{ triangle $X \overset{f}\to Y \to Z \to X[1]$ in $\T{}$ \st $Z \in \clL{}$} \}. $$
It is a well known fact that the Verdier quotient always has
coproducts, admits a natural triangulated structure, and the
canonical localization functor $Q: \T \to \T/\clL$ is exact and preserves
coproducts, \cite[Chapter 2]{N}. However, one has to be careful, since
$\T/\clL$ might not be a usual category in the sense that the
homomorphism spaces might be proper classes rather than sets. This fact,
although often inessential and neglected, as $\T/\clL$ has a very
straightforward and constructive description, may nevertheless have
important consequences in some cases; see eg.~\cite{CN}.

Let $L: \T \to \T$ be an exact endofunctor of $\T$. Then $L$ is called
a \emph{localization functor} if there exists a natural transformation
$\eta: \mathrm{Id}_\T \to L$ \st $L\eta_X = \eta_{LX}$ and $\eta_{LX}:
LX \to L^2 X$ is an isomorphism for each $X \in \T$.

It is easy to check that the full subcategory $\Ker L$
of $\T$ given by
$$ \Ker L = \{ X \in \T \mid LX = 0 \} $$
is always localizing \cite[1.2]{AJS}. Moreover, there is a canonical
triangle equivalence between $\T/\Ker L$ and $\Img L$, the essential
image of $L$; see \cite[9.1.16]{N} or \cite[4.9.1]{Kr}.
This among other things implies that all morphism spaces in $\T/\Ker L$ are
sets. Note that although $\Img L$ has coproducts as a category, it
might \emph{not} be closed
under coproducts in $\T$. This type of localization, coming from a
localization functor, is often referred to as \emph{Bousfield localization}.
However, not every localizing
subcategory $\clL$ is realized as the kernel of a localization
functor, \cite[1.3]{CN}. Namely, $\clL$ is of the form $\Ker L$ for
some localization functor \iff the inclusion $\clL \to \T$ has a right adjoint, \cite[1.6]{AJS}.

\smallskip

A central concept in this paper is that of a well generated
triangulated category. Let $\kappa$ be a regular cardinal number. An
object $Y$ in a category with arbitrary coproducts is called
\emph{$\kappa$--small} provided that every morphism of the form
$$ Y \longrightarrow \coprod_{i \in I} X_i $$
factorizes through a subcoproduct $\coprod_{i \in J} X_i$ with $\lvert
J \rvert < \kappa$.

\begin{defn} \label{defn:well_gen}
Let $\T$ be a triangulated category with arbitrary coproducts and
$\kappa$ be a regular cardinal. Then $\T$ is called
\emph{$\kappa$--well generated} provided there is a set $\clS$ of
objects of $\T$ satisfying the following conditions:
\begin{enumerate}
\item If $X \in \T$ \st $\T(Y,X) = 0$ for each $Y \in \clS$, then
$X = 0$;
\item Each object $Y \in \clS$ is $\kappa$--small;
\item For any morphism in $\T$ of the form $f: Y \to \coprod_{i \in I}
X_i$ with $Y \in \clS$, there exists a family of morphisms
$f_i: Y_i \to X_i$ \st $Y_i \in \clS$ for each $i \in I$ and $f$ factorizes as
$$
Y \xrightarrow{\phantom{\coprod f_i}}
\coprod_{i \in I} Y_i \xrightarrow{\coprod f_i}
\coprod_{i \in I} X_i.
$$
\end{enumerate}

The category $\T$ is called \emph{well generated} if it is
$\kappa$--well generated for some regular cardinal $\kappa$.
\end{defn}

This definition differs to some extent from Neeman's original
definition in~\cite[8.1.7]{N}. The equivalence between the two follows
from~\cite[Theorem A]{Kr6} and~\cite[Lemmas 4 and 5]{Kr6}.
Note that if $\kappa = \aleph_0$, then condition $(3)$ is
vacuous and $\aleph_0$--well generated triangulated categories are precisely
the \emph{compactly generated} triangulated categories in the usual sense.

The key property of well generated categories is that the Brown
Representability Theorem holds:
\begin{prop} \label{prop:brown} \cite[8.3.3]{N}
Let $\T$ be a well generated triangulated category. Then:
\begin{enumerate}
\item Any contravariant cohomological functor $F: \T \to \Ab$ which takes
coproducts to products is, up to isomorphism, of the form $\T(-,X)$
for some $X \in \T$.
\item If $\clS$ is a set of objects of $\T$ which meets
  assumptions $(1)$, $(2)$ and $(3)$ of Definition~\ref{defn:well_gen}
  for some cardinal $\kappa$, then $\T = \Loc \clS$.
\end{enumerate}
\end{prop}

\smallskip

Next we turn our attention to categories of complexes. Let $\B$ be an additive category. Using a standard notation, we denote by $\Comp(\B)$ the category of chain complexes
$$
\cpx X: \qquad
\dots \to X^{n-1} \overset{d^{n-1}}\to X^n \overset{d^n}\to X^{n+1} \to \dots,
\qquad
$$
of objects of $\B$. By $\Htp(\B)$, we denote the factor-category of
$\Comp(\B)$ modulo the ideal of null-homotopic chain complex
morphisms. It is well known that $\Htp(\B)$ has a triangulated
structure where triangle completions are constructed using mapping
cones (see for example~\cite[Chapter I]{H}). Moreover, if $\B$ has
arbitrary coproducts, so have them both $\Comp(\B)$ and $\Htp(\B)$,
and the canonical functor $\Comp(\B) \to \Htp(\B)$ preserves
coproducts.

We will often take for $\B$ module categories or their subcategories.
In this case, $R$ will denote an associative unital ring and
$\ModR$ the category of all (unital) right $R$--modules. By $\ProjR$ and
$\FlatR$ we denote, respectively, the full subcategories of projective
and flat $R$--modules.

In fact, our considerations will usually work in a more general
setting. Let $\A$ be a skeletally small additive category and
$\ModA$ be the category of all contravariant additive functors
$\A \to \Ab$. We will call such functors \emph{right $\A$--modules}.
Then $\ModA$ shares many formal properties with usual module
categories. We refer to~\cite[Appendix B]{JL} for more
details. Correspondingly, we denote by $\ProjA$ the full
subcategory of projective functors and by $\FlatA$ the category of
flat functors. We discuss the categories of the form $\FlatA$
more in detail in Section~\ref{sec:fin_acc} since those are, up to
equivalence, precisely the so called additive finitely accessible categories. Many
natural abelian categories are of this form.

\smallskip

Finally, we spend a few words on set-theoretic considerations. All our
proofs work in ZFC with an extra technical assumption: the axiom of
choice for proper classes. The latter assumption has no algebraic
significance, it is only used to keep arguments simple in the following
case:

Let $F: \C \to \D$ be a covariant additive functor. If
we know, for example by
the Brown Representability Theorem, that the composition of functors
$$
\C \xrightarrow{\;\;\;\;\, F \;\;\;\;\,}
\D \xrightarrow{\D(-,X)} \Ab
$$
is representable for each $X \in \D$, we would like to conclude that
$F$ has a right adjoint $G: \D \to \C$. In order to do that, we must
for each $Y \in \C$ \emph{choose} one particular value for $GY$ from a class
of mutually isomorphic candidates.

\section{Pure semisiplicity}
\label{sec:pss}

A relatively straightforward but crucial obstacle causing a homotopy
category of complexes $\Htp(\B)$ not to be well generated is that the
additive base category $\B$ is not pure semisimple. Here, we use the
following very general definition:

\begin{defn} \label{defn:pss}
An additive category $\B$ with arbitrary coproducts is called
\emph{pure semisimple} if it has an additive generator. That is,
there is an object $X \in \B$ \st $\B = \Add X$, where $\Add X$
stands for the full subcategory formed by all objects which are summands in (possibly
infinite) coproducts of copies of $X$.
\end{defn}

The term is inspired by the case $\B = \ModR$, where we have the
following proposition:

\begin{prop} \label{prop:pss_ModR}
A ring $R$ is right pure semisimple (that is, each pure monomorphism
between right $R$--modules splits) \iff $\ModR$ is pure
semisimple in the sense of Definition~\ref{defn:pss}.
\end{prop}

\begin{proof}
If every pure monomorphism in $\ModR$ splits, then also every pure
epimorphism splits. That is, every module is pure projective, or
equivalently a summand in a direct sum
of finitely presented modules. By a theorem of Kaplansky,
\cite[Theorem 1]{Kap}, it follows that every module is a direct sum of
countably generated modules. Hence, $\ModR$ is pure semisimple
according to our definition. In fact, one can show more in this case:
Every module is even a direct sum of finitely presented modules; see for
example~\cite{Z-H} or~\cite[App. B]{JL}.


Let us conversely assume that $\ModR$ is a pure semisimple additive category.
Using~\cite[Theorem 26.1]{AF}, which is a variation of~\cite[Theorem 1]{Kap}
for higher cardinalities, we see that if
$\ModR = \Add X$ for some $\kappa$--generated module $X$, then
each module in $\ModR$ is a direct sum of
$\lambda$--generated modules where $\lambda = \max(\kappa,\aleph_0)$.
This fact implies that every module is $\Sigma$--pure
injective, \cite{GJ}. In particular, each pure monomorphism in $\ModR$
splits and $R$ is right pure semisimple.
\end{proof}

If $R$ is an artin algebra, then the conditions of
Proposition~\ref{prop:pss_ModR} are well-known to
be further equivalent to $R$ being of finite representation
type; see~\cite[Theorem A]{Aus}. For more details and references on
this topic, we also refer to~\cite{Z-H}.
It turns out that the pure semisimplicity condition has a nice
interpretation for finitely accessible additive categories as well.
We will discuss this more in detail in Section~\ref{sec:fin_acc}.

\smallskip

For giving a connection between pure semisimplicity of $\B$ and properties of $\Htp(\B)$, we recall a structure result for the so-called contractible complexes in $\Comp(\B)$. A complex $\cpx Y \in \Comp(\B)$ is \emph{contractible} if it is mapped to a zero object under $\Comp(\B) \to \Htp(\B)$. It is clear that the complexes of the form
$$
\cpx I_{X,n}: \qquad
\dots \to 0 \to 0 \to X = X \to 0 \to 0 \to \dots,
\qquad
$$
such that the first $X$ is in degree $n$, are contractible. Moreover, all other contractible complexes are obtained in the following way:

\begin{lem} \label{lem:htp_zero}
Let $\B$ be an additive category with splitting idempotents and $\cpx Y \in \Comp(\B)$. Then the following are equivalent:
\begin{enumerate}
\item $\cpx Y$ is contractible;
\item $\cpx Y$ is isomorphic in $\Comp(\B)$ to a complex of the form $\coprod_{n \in \Z} \cpx I_{X_n,n}$.
\end{enumerate}
\end{lem}

\begin{proof}
$(2) \implies (1)$. This is trivial given the fact that the functor $\Comp(\B) \to \Htp(\B)$ preserves those componentwise coproducts of complexes which exist in $\Comp(\B)$.

$(1) \implies (2)$. Let us fix a contractible complex in $\Htp(\B)$:
$$
\cpx Y: \qquad
\dots \xrightarrow{d^{n-2}} 
Y^{n-1} \xrightarrow{d^{n-1}}
Y^n \xrightarrow{\;\;d^n\;\,}
Y^{n+1} \xrightarrow{d^{n+1}}
\dots.
\qquad
$$
By definition, the identity morphism of $\cpx Y$ is homotopy equivalent to the zero morphism in $\Comp(\B)$, so there are morphisms $s^n: Y^n \to Y^{n-1}$ in $\B$ \st
$$ 1_{Y^n} = d^{n-1} s^n + s^{n+1} d^n. $$
When composing with $d^n$, we get $d^n = d^n s^{n+1} d^n$, so
$s^{n+1} d^n: Y^n \to Y^n$ is idempotent in $\B$ for each $n \in \Z$. Hence
there are morphisms $p^n: Y^n \to X_n$ and $j^n: X_n \to Y^n$ in $\B$ \st
$p^n j^n = 1_{X_n}$ and $j^n p^n = s^{n+1} d^n$. Let us denote by
$f^n: X_{n-1} \amalg X_n \to Y^n$ and $g^n: Y^n \to X_{n-1} \amalg X_n$ the
morphisms defined as follows:
$$
f^n = (d^{n-1} j^{n-1},j^n), \quad \textrm{ and } \quad
g^n = \begin{pmatrix} p^{n-1} s^n \\ p^n \end{pmatrix}.
$$
Using the identities above, it is easy to check that $f^n g^n =
1_{Y^n}$ and $g^n f^n$ is an isomorphism in $\B$ for each
$n$. Therefore, both $f^n$ and $g^n$ are isomorphisms and $g^n f^n$ is
the identity morphism. Finally, it is
straightforward to check that the family of morphisms $(f_n \mid n \in
\Z)$ induces an (iso)morphism $f: \coprod_{n \in Z} \cpx I_{X_n,n} \to
\cpx Y$ in $\Comp(\B)$.
%
%
%
\end{proof}

It is not difficult to see that the condition of $\B$ having splitting
idempotents is really necessary in Lemma~\ref{lem:htp_zero}. However,
there is a standard construction which allows us to amend $\B$ with
the missing summands if $\B$ does not have splitting idempotents.

\begin{defn} \label{defn:idemp}
Let $\B$ be an additive category. Then an additive category $\bar\B$
is called an \emph{idempotent completion} of $\B$ if
\begin{enumerate}
\item $\bar\B$ has splitting idempotents;
\item $\B$ is a full subcategory of $\bar\B$;
\item Every object in $\bar\B$ is a direct summand of an object in $\B$.
\end{enumerate}
\end{defn}

It is a classical result that idempotent completions always exist. We
refer for example to~\cite[\S 1]{BS} for a particular
construction. Moreover, it is well-known that if $\B$ has arbitrary
coproducts, then also $\bar\B$ has them and they are compatible with
coproducts in $\B$.

Now we can state the main result of the section showing that for $\Htp(\B)$ being generated by a set (and, in particular, for $\Htp(\B)$ being well generated), the category $\B$ is necessarily pure semisimple.

\begin{thm} \label{thm:gen_by_set}
Let $\B$ be an additive category with arbitrary coproducts and assume that there is a set of objects $\clS \subseteq \Htp(\B)$ \st $\Htp(\B) = \Loc \clS$. Then $\B$ is pure semisimple.
\end{thm}

\begin{proof}
Note that we can replace $\clS$ by a singleton $\{\cpx Y\}$; take for instance $\cpx Y = \coprod_{\cpx Z \in \clS} \cpx Z$. Let us denote by $X \in \B$ the coproduct $\coprod_{n \in \Z} Y^n$ of all components of $\cpx Y$. We will show that $\B = \Add X$.
First, we claim that $\Htp(\Add X)$ is a dense subcategory of
$\Htp(\B)$, that is, each object in $\Htp(\B)$ is isomorphic to one in $\Htp(\Add X)$.
Indeed, $\cpx Y \in \Htp(\Add X)$ and one easily checks that the
closure of $\Htp(\Add X)$ under taking isomorphic objects in
$\Htp(\B)$ is a localizing subcategory. Hence
$\Htp(\Add X)$ is dense in $\Htp(\B)$ and the claim is proved.

Suppose for the moment that $\B$ has splitting idempotents. If we identify $\B$ with the full subcategory of $\Htp(\B)$ formed by complexes concentrated in degree zero, we have proved that each object $Z \in \B$ is isomorphic to a complex $\cpx Q \in \Htp(\Add X)$. That is, there is a chain complex homomorphism $f: Z \to \cpx Q$ \st $\cpx Q \in \Comp(\Add X)$ and $f$ becomes an isomorphism in $\Htp(\B)$. In particular, the mapping cone $C_f$ of $f$ is contractible:
$$
C_f: \qquad
\dots \longrightarrow Q^{-3} \overset{d^{-3}}\longrightarrow Q^{-2}
\overset{( ^{d^{-2}}_{\;\;\,0} )}\longrightarrow Q^{-1} \amalg Z
\overset{( d^{-1}, f^0 )}\longrightarrow Q^0
\overset{d^0}\longrightarrow Q^1 \longrightarrow \dots
\qquad
$$
Here, $f^0$ is the degree $0$ component of $f$. Consequently, Lemma~\ref{lem:htp_zero} yields the following commutative diagram in $\B$ with isomorphisms in columns:
$$
\begin{CD}
Q^{-2}  @>{( ^{d^{-2}}_{\;\;\,0} )}>>   Q^{-1} \amalg Z @>{( d^{-1}, f^0 )}>>   Q^0     \\
@V{\cong}VV             @V{\cong}VV             @V{\cong}VV \\
U \amalg V  @>{( ^{0\;1}_{0\;0} )}>>    V \amalg W   @>{( ^{0\;1}_{0\;0} )}>>   W \amalg Z  \\
\end{CD}
$$
It follows that $V,W$ and also $Q^{-1} \amalg Z$ and $Z$ are in $\Add X$. Hence $\B = \Add X$.

Finally, let $\B$ be a general additive category with coproducts and
$\bar\B$ be its idempotent completion. From the fact that $\Htp(\B)$
has splitting idempotents, \cite[1.6.8]{N}, one easily sees that the
full embedding $\Htp(\B) \to \Htp(\bar\B)$ is dense. We already know
that if $\Htp(\B) = \Loc\clS$ for a set $\clS$, then $\bar\B = \Add X$
for some $X \in \bar\B$. In fact, we can take $X \in \B$ by the above
construction. But then clearly $\B = \Add X$ when the additive closure
is taken in $\B$. Hence $\B$ is pure semisimple.
\end{proof}

\begin{rem}
When studying well generated triangulated categories, an important role is played by so-called $\kappa$--localizing subcategories, see \cite{N,Kr}. We recall that given a cardinal number $\kappa$, a \emph{$\kappa$--coproduct} is a coproduct with fewer than $\kappa$ summands. If $\T$ is a triangulated category with arbitrary $\kappa$--coproducts, a thick subcategory $\clL \subseteq \T$ is called \emph{$\kappa$--localizing} if it is closed under taking $\kappa$--coproducts. In this context, one can state the following ``bounded'' version of Theorem~\ref{thm:gen_by_set}:

Let $\kappa$ be an uncountable regular cardinal and $\B$ be an additive category with $\kappa$--coproducts. If $\Htp(\B)$ is generated as a $\kappa$--localizing subcategory by a set $\clS$ of fewer than $\kappa$ objects, then there is $X \in \B$ \st every object of $\B$ is a summand in a $\kappa$--coproduct of copies of $X$.

%
\end{rem}

Note that Theorem~\ref{thm:gen_by_set} gives immediately a wide range
of examples of categories which are not well generated. For instance,
$\Htp(\ModR)$ is not well generated for any ring $R$ which is not
right pure semisimple. One can take $R = \Z$ or
$R = k(\cdot \!\rightrightarrows\! \cdot)$,
the Kronecker algebra over a field $k$. The fact that $\Htp(\Ab)$
is not well generated was first observed
by Neeman, \cite[E.3.2]{N}, using different arguments. In fact, we can
state the following proposition, which we later generalize in
Section~\ref{sec:well-gen}:

\begin{prop} \label{prop:well_gen_rings}
Let $R$ be a ring. Then the following are equivalent:
\begin{enumerate}
\item $\Htp(\ModR)$ is well generated;
\item $\Htp(\ModR)$ is compactly generated;
\item $R$ is right pure semisimple.
\end{enumerate}
If $R$ is an artin algebra, the conditions are further equivalent to:
\begin{enumerate}
\item[(4)] $R$ is of finite representation type.
\end{enumerate}
\end{prop}

\begin{proof}
$(2) \implies (1)$ is clear, as compactly generated is the same as
$\aleph_0$--well generated.
$(1) \implies (3)$ follows by Theorem~\ref{thm:gen_by_set}
and Proposition~\ref{prop:pss_ModR}.
$(3) \implies (2)$ has been proved by Holm and J\o{}rgensen,
\cite[\S 4 (3), p. 17]{HJ}.
Finally, the equivalence between $(3)$ and $(4)$ is due to
Auslander, \cite[Theorem A]{Aus}.
\end{proof}


\section{Locally well generated triangulated categories}
\label{sec:loc-well-gen}

We have seen in the last section that a triangulated category of the
form $\Htp(\ModR)$ is often not well generated. One might get an
impression that handling such categories is hopeless, but the main
problem here actually is that the category is very big in the sense
that it is not generated by any set. Otherwise, it has a very
reasonable structure. We shall see that it is locally well generated
in the following sense:

\begin{defn} \label{defn:loc_well_gen}
A triangulated category $\T$ with arbitrary coproducts is called \emph{locally well generated} if $\Loc\clS$ is well generated for any set $\clS$ of objects of $\T$.
\end{defn}

In fact, we prove that $\Htp(\ModA)$ is locally well generated for any
skeletally small additive category $\A$. To this end, we first need
to be able to measure the size of modules and complexes.

\begin{defn} \label{defn:card}
Let $\A$ be a skeletally small additive category and $M \in
\ModA$. Recall that $M$ is a contravariant additive functor $\A \to \Ab$
by definition. Then the \emph{cardinality of $M$}, denoted by
$|M|$, is defined as
$$ |M| = \sum_{A \in \clS} |M(A)|, $$
where $|M(A)|$ is just the usual cardinality of the group $M(A)$ and
$\clS$ is a fixed representative set for isomorphism classes of
objects from $\A$. The \emph{cardinality of a complex
$\cpx Y = (Y^n,d^n) \in \Htp(\ModA)$} is defined as
$$ |\cpx Y| = \sum_{n \in \Z} |Y^n|. $$
\end{defn}

It is not so difficult to see that the category of all complexes whose
cardinalities are bounded by a given regular cardinal always gives rise to a
well-generated subcategory of $\Htp(\ModA)$:

\begin{lem} \label{lem:perfect}
Let $\A$ be a skeletally small additive category and $\kappa$ be an infinite cardinal. Then the full subcategory $\clS_\kappa$ formed by all complexes of cardinality less than $\kappa$ meets conditions $(2)$ and $(3)$ of Definition~\ref{defn:well_gen}.

In particular, $\T_\kappa = \Loc \clS_\kappa$ is a $\kappa$--well generated subcategory of $\Htp(\ModA)$ for any regular cardinal $\kappa$.
\end{lem}

\begin{proof}
Let $\cpx Y \in \Htp(\ModA)$ \st $|\cpx Y| < \kappa$. If $(\cpx Z_i \mid i \in I)$ is an arbitrary family of complexes in $\Htp(\ModA)$, we can construct their coproduct as a componentwise coproduct in $\Comp(\ModA)$. Then whenever $f: \cpx Y \to \coprod_{i \in I} \cpx Z_i$ is a morphism in $\Comp(\ModA)$, it is straightforward to see that $f$ factorizes through $\coprod_{i \in J} \cpx Z_i$ for some $J \subseteq I$ of cardinality less than $\kappa$. Hence $\cpx Y$ is $\kappa$--small in $\Htp(\ModA)$.

Regarding part $(3)$ of Definition~\ref{defn:well_gen}, consider a
morphism $f: \cpx Y \to \coprod_{i \in I} \cpx Z_i$. We have the
following factorization in the abelian category of complexes $\Comp(\ModA)$:
$$
\cpx Y \overset{(f_i)}\longrightarrow \coprod_{i \in I} \Img f_i \overset{j}\longrightarrow
\coprod_{i \in I} \cpx Z_i.
$$
Here, $f_i: \cpx Y \to \cpx Z_i$ are the compositions of $f$ with the
canonical projections $\pi_i: \coprod_{i' \in I} \cpx Z_{i'} \to \cpx Z_i$,
and $j$ stands for the obvious inclusion. It is easy to see that $|\Img f_i|
< \kappa$ for each $i \in I$ and that the morphism $j$ is a coproduct
of the inclusions $\Img f_i \to \cpx Z_i$. Hence $(3)$ is satisfied.

For the second part, let $\kappa$ be regular and $\T_\kappa = \Loc
\clS_\kappa$. Let us denote by $\clS'$ a representative set of
objects in $\clS_\kappa$. It only remains to prove that $\clS'$
satisfies condition $(1)$ of Definition~\ref{defn:well_gen}, which is
rather easy. Namely, let $\cpx X \in \T_\kappa$ \st
$\T_\kappa(\cpx Y,\cpx X) = 0$ for each $\cpx Y \in \clS'$. Then
$\T' = \{\cpx Y \in \T_\kappa \mid \T_\kappa(\cpx Y,\cpx X) = 0 \}$
defines a localizing subcategory of $\T_\kappa$ containing
$\clS_\kappa$. Hence, $\T' = \T_\kappa$ and $\cpx X = 0$.
\end{proof}

We will also need (a simplified version of) an important result,
which is essentially contained already in~\cite{N}.
It says that the property of being well generated is preserved when
passing to any localizing subcategory generated by a set.
In particular, every well generated category is locally well generated.

\begin{prop} \cite[Theorem 7.2.1]{Kr} \label{prop:well_inherit}
Let $\T$ be a well generated triangulated category and $\clS \subseteq \T$ be a set of objects. Then $\Loc\clS$ is a well generated triangulated category, too.
\end{prop}

Now, we are in a position to state a theorem which gives us a major
source of examples of locally well generated triangulated categories.

\begin{thm} \label{thm:K_loc_well_gen}
Let $\A$ be a skeletally small additive category. Then the
triangulated category $\Htp(\ModA)$ is locally well generated.
\end{thm}

\begin{proof}
As in Lemma~\ref{lem:perfect}, we denote by $\clS_\kappa$ the
full subcategory of $\Htp(\ModA)$ formed by complexes of cardinality
less than $\kappa$ and put $\T_\kappa = \Loc \clS_\kappa$, the
localizing class generated by $\clS_\kappa$ in $\Htp(\ModA)$.
Then $\T_\kappa$ is
($\kappa$--)well generated for each regular cardinal $\kappa$ by
Lemma~\ref{lem:perfect} and clearly
$$
\Htp(\ModA) = \bigcup_{\kappa \textrm{ regular}} \clS_\kappa =
\bigcup_{\kappa \textrm{ regular}} \T_\kappa.
$$
Now, if $\clS \subseteq \Htp(\ModA)$ is a set of objects, then
$\clS \subseteq \T_\kappa$ for some $\kappa$. Hence also
$\Loc\clS \subseteq \T_\kappa$ and $\Loc\clS$ is well generated by
Proposition~\ref{prop:well_inherit}. It follows that $\Htp(\ModA)$ is
locally well generated.

\end{proof}

Having obtained a large class of examples of locally well generated
triangulated categories, one might ask for some basic properties of
such categories. We will prove a version of the so-called Bousfield
Localization Theorem here:

\begin{prop} \label{prop:bousfield}
Let $\T$ be a locally well generated triangulated category and $\clS \subseteq \T$ be a
set of objects. Then $\T/\Loc\clS$ is a Bousfield localization; that
is, there is a localization functor $L: \T \to \T$ \st
$\Ker L = \Loc\clS$. In particular, we have
$$ \Img L = \{ X \in \T \mid \T(Y,X) = 0 \textrm{ for each } Y \in \clS \}, $$
there is a canonical triangle equivalence between $\T/\Loc\clS$ and $\Img L$
given by the composition
$$
\Img L \overset{\subseteq}\longrightarrow
\T \overset{Q}\longrightarrow \T/\Loc\clS,
$$
and all morphism spaces in $\T/\Loc\clS$ are sets.
\end{prop}

\begin{proof}
The proof is rather standard.
$\Loc\clS$ is well generated, so it satisfies the
Brown Representability Theorem (see Proposition~\ref{prop:brown}).
Hence the inclusion $\mathbf{i}: \Loc\clS \to \T$ has a right adjoint
by~\cite[8.4.4]{N}. The composition of this right adjoint with
$\mathbf{i}$ gives a so-called colocalization functor
$\Gamma: \T \to \T$ whose essential image is equal to $\Loc\clS$.
The definition of a colocalization functor is formally dual to the one
of a localization functor; see~\cite[\S 4.12]{Kr} for details. 
A well-known construction then yields a localization functor
$L: \T \to \T$ \st $\Ker L = \Loc\clS$. We refer to~\cite[9.1.14]{N}
or~\cite[4.12.1]{Kr} for details. The rest follows from
\cite[9.1.16]{N} or \cite[4.9.1]{Kr}.
\end{proof}

\begin{rem}
Proposition~\ref{prop:bousfield} has been proved before for well
generated triangulated categories. This is implicitly contained for
example in~\cite[\S 7.2]{Kr}. It also generalizes more
classical results, such as a corresponding statement for the derived
category $\Der(\B)$ of a Grothendieck abelian category $\B$,
\cite[5.7]{AJS}. To see this, one only needs to observe that $\Der(\B)$
is well generated, see~\cite[Example 7.7]{Kr}.
\end{rem}

An obvious question is whether the Brown Representability Theorem also
holds for locally well generated categories, as this was the crucial
feature of well generated categories. Unfortunately, this is not the
case in general, as the following example suggested by Henning Krause
shows.

\begin{expl} \label{expl:brow_rep_failure}
According to~\cite[Exercise 1, p. 131]{F}, one can construct an
abelian category $\B$ with some $\Ext$-spaces being proper
classes. Namely, let $U$ be the class of all cardinals, and let $\B =
\Modu{\Z\spanned{U}}$, the category of all ``modules over the free ring on
the proper class of generators $U$.'' That is, an object $X$ of $\B$ is an
abelian group \st each $\kappa \in U$ has a $\Z$-linear action on $X$
and this action is trivial for all but a set of cardinals. Such a
category admits a valid set-theoretical description in ZFC. If we
denote by $\Z$ the object of $\B$ whose underlying group is free of
rank $1$ and $\kappa \cdot \Z = 0$ for each $\kappa \in U$, then
$\Ext^1_\B(\Z,\Z)$ is a proper class (see also~\cite[4.15]{Kr}
or~\cite[1.1]{CN}).

Given the above description of objects of $\B$, one can easily adjust
the proof of Theorem~\ref{thm:K_loc_well_gen} to see that $\Htp(\B)$
is locally well generated. Let $\Htp_\mathrm{ac}(\B)$ stand for the full
subcategory of all acyclic complexes in $\Htp(\B)$. Then
$\Htp_\mathrm{ac}(\B)$ is clearly a localizing subcategory of
$\Htp(\B)$, hence locally well-ge\-ne\-ra\-ted.

It has been shown in~\cite{CN} that $\Htp_\mathrm{ac}(\B)$ does not
satisfy the Brown Representability Theorem. In fact, one proved even
more: $\Htp_\mathrm{ac}(\B)$ is localizing in $\Htp(\B)$, but it is
not a kernel of any localization functor $L: \Htp(\B) \to \Htp(\B)$.
More specifically, the composition of functors, the second of which
is contravariant,
$$
\begin{CD}
\Htp_\mathrm{ac}(\B)  @>{\subseteq}>>  \Htp(\B)  @>{\Htp(\B)(-,\Z)}>>  \Ab
\end{CD}
$$
is not representable by any object of $\Htp_\mathrm{ac}(\B)$.
\end{expl}

Yet another natural question is what other triangulated categories are
locally well generated. A deeper analysis of this problem is left for
future research, but we will see in Section~\ref{sec:fin_acc} that
$\Htp(\B)$ is locally well generated for any finitely accessible
additive category $\B$. For now, we will
prove that the class of locally well generated triangulated categories
is closed under some natural constructions. Let us start with a general
lemma, which holds even if morphism spaces in the quotient $\T/\clL$
are proper classes:

\begin{lem} \label{lem:loc}
Let $\T$ be a triangulated category and $\clL \subseteq \clL'$ be two
localizing subcategories of $\T$. Then $\clL'/\clL$ is a localizing
subcategory of $\T/\clL$.
\end{lem}

\begin{proof}
It is easy to see that $\clL'/\clL$ is a full subcategory of $\T/\clL$ which is closed under taking isomorphic objects, see~\cite[Th\'eor\`eme 4-2]{V} or~\cite[Proposition 1.6.5]{KaSch}. The rest follows directly from the construction of $\T/\clL$.
\end{proof}

Now we can show that taking localizing subcategories and localizing
with respect to a set of objects preserves the locally well generated
property.

\begin{prop} \label{prop:loc_well_inherit}
Let $\T$ be a locally well generated triangulated category.
\begin{enumerate}
\item Any localizing subcategory $\clL$ of $\T$ is itself locally well
  generated.
\item The Verdier quotient $\T/\Loc\clS$ is locally well generated for
  any set $\clS$ of objects in $\T$.
\end{enumerate}
\end{prop}

\begin{proof}
$(1)$ is trivial. For $(2)$, put $\clL = \Loc\clS$ and consider a set
$\C$ of objects in $\T/\clL$. We have to prove that the localizing
subcategory generated by $\C$ in $\T/\clL$ is well generated.
Since the objects of $\T$ and $\T/\clL$ coincide by definition, we can
consider a localizing subcategory $\clL' \subseteq \T$ defined by
$\clL' = \Loc (\clS \cup \C)$. One easily sees using
Lemma~\ref{lem:loc} that $\clL'/\clL = \Loc\C$ in $\T/\clL$. Since
both $\clL$ and $\clL'$ are well generated by definition, so is
$\clL'/\clL$ by~\cite[7.2.1]{Kr}. Hence $\T/\clL$ is locally well
generated.
\end{proof}

We conclude this section with an immediate consequence of
Theorem~\ref{thm:K_loc_well_gen} and
Proposition~\ref{prop:loc_well_inherit}, which will be useful in the
next section:

\begin{cor} \label{cor:add_well_gen}
Let $\A$ be a small additive category and $\B$ be a full subcategory of $\ModA$ which is closed under arbitrary coproducts. Then $\Htp(\B)$ is locally well generated.
\end{cor}

\section{Finitely accessible additive categories}
\label{sec:fin_acc}

There is a natural generalization of module categories, namely the additive version of finitely accessible categories in the terminology of~\cite{AR}. As we have seen, there is quite a lot of freedom to choose $\B$ in the above Corollary~\ref{cor:add_well_gen}. We will use this fact and a standard trick to (seemingly) generalize Theorem~\ref{thm:K_loc_well_gen} from module categories to finitely accessible additive categories. We start with a definition.

\begin{defn} \label{defn:fin_acc}
Let $\B$ be an additive category which admits arbitrary filtered colimits. Then:
\begin{itemize}
\item An object $X \in \B$ is called \emph{finitely presentable} if the representable functor $\B(X,-): \B \to \Ab$ preserves filtered colimits.
\item The category $\B$ is called \emph{finitely accessible} if there is a set $\A$ of finitely presentable objects from $\B$ \st every object in $\B$ is a filtered colimit of objects from $\A$.
\end{itemize}
\end{defn}

Note that if $\B$ is finitely accessible, the full subcategory $\fpB$
of $\B$ formed by all finitely presentable objects in $\B$ is
skeletally small, \cite[2.2]{AR}. Several other general properties
of finitely accessible categories will follow from
Proposition~\ref{prop:repres}.

Finitely accessible categories occur at many occasions. The simplest
and most natural example is the module category $\ModR$ over an
associative unital ring. It is well-known that finitely presentable
objects in $\ModR$ coincide with finitely presented $R$--modules in the
usual sense. The same holds for $\ModA$, the
category of modules over a small additive category $\A$.
Motivated by representation theory, finitely
accessible categories were studied by Crawley-Boevey~\cite{CB}
under the name locally finitely presented
categories; see~\cite[\S 5]{CB} for further examples.
The term from~\cite{CB}, however, may cause some confusion in the light
of other definitions. Namely, Gabriel and Ulmer~\cite{GU} have defined
the concept
of a \emph{locally finitely presentable} category which is, in our
terminology, a cocomplete finitely accessible category. As the latter concept
has been used quite substantially in one of our main references, \cite{Kr},
we stick to the terminology of~\cite{AR}.

The crucial fact about finitely accessible additive categories is the following representation theorem:

\begin{prop} \label{prop:repres}
The assignments
$$ \A \mapsto \FlatA \quad \textrm{ and } \quad \B \mapsto \fpB $$
form a bijective correspondence between
\begin{enumerate}
\item equivalence classes of skeletally small additive categories $\A$ with splitting idempotents, and
\item equivalence classes of additive finitely accessible categories $\B$.
\end{enumerate}
\end{prop}

\begin{proof}
See \cite[\S 1.4]{CB}.
\end{proof}

\begin{rem}
The correspondence from Proposition~\ref{prop:repres} restricts, using~\cite[\S 2.2]{CB},
to a bijection between equivalence classes of skeletally small
additive categories with finite colimits (equivalently, with cokernels) and
equivalence classes of locally finitely presentable
categories in the sense of Gabriel and Ulmer~\cite{GU}.
\end{rem}

One of the main results of this paper has now become a mere corollary of preceding results:

\begin{thm} \label{thm:fin_acc_loc_well_gen}
Let $\B$ be a finitely accessible additive category. Then $\Htp(\B)$ is locally well generated.
\end{thm}

\begin{proof}
Let us put $\A = \fpB$, the full subcategory of $\B$ formed by all finitely presentable objects. Using Proposition~\ref{prop:repres}, we see that $\B$ is equivalent to the category $\FlatA$. The category $\Htp(\FlatA)$ is locally well generated by Corollary~\ref{cor:add_well_gen}, and so must be $\Htp(\B)$.
\end{proof}

The remaining question when $\Htp(\B)$ is $\kappa$--well generated and which cardinals $\kappa$ can occur will be answered in the next section. For now, we know by Theorem~\ref{thm:gen_by_set} that a necessary condition is that $\B$ be pure semisimple. In fact, we will show that this is also sufficient, but at the moment we will only give a better description of pure semisimple finitely accessible additive categories.

\begin{prop} \label{prop:pss_fin_acc}
Let $\B$ be a finitely accessible additive category. Then the following are equivalent:
\begin{enumerate}
\item $\B$ is pure semisimple in the sense of Definition~\ref{defn:pss};
\item Each object in $\B$ is a coproduct of (indecomposable) finitely
  presentable objects;
\item Each flat right $\A$--module is projective, where $\A = \fpB$.
\end{enumerate}
\end{prop}

\begin{proof}
For the whole argument, we put $\A = \fpB$ and without loss of
generality assume that $\B = \FlatA$.

$(1) \implies (3)$. Assume that $\FlatA$ is pure semisimple. As in the
proof for Proposition~\ref{prop:pss_ModR}, we can use a generalization~\cite[Theorem 26.1]{AF} of Kaplansky's theorem,
to deduce that there is a cardinal number $\lambda$ \st
each flat $\A$--module is a direct sum of at most
$\lambda$--generated flat $\A$--modules. The key step is then
contained in~\cite[Corollary 3.6]{GIT} which says that under the
latter condition $\A$ is a right perfect category. That is, it
satisfies the equivalent conditions of Bass'
theorem~\cite[B.12]{JL} (or more precisely, its version for contravariant
functors $\A \to \Ab$). One of the equivalent conditions is condition
$(3)$.

$(3) \implies (2)$. This is a consequence of Bass' theorem;
see~\cite[B.13]{JL}.

$(2) \implies (1)$. Trivial, $\B = \Add X$ where
$X = \bigoplus_{Y \in \A} Y$.
\end{proof}

For further reference, we mention one more condition which
one might impose on a finitely accessible additive category.
Namely, it is well known that for a ring $R$, the category $\FlatR$ is
closed under products
\iff $R$ is left coherent. This generalizes in a natural way for
finitely accessible additive categories. Let us recall that an
additive category $\A$ is said to have \emph{weak cokernels} if
for each morphism $X \to Y$ there is a morphism $Y \to Z$ \st
$\A(Z,W) \to \A(Y,W) \to \A(X,W)$ is exact for all $W \in \A$.

\begin{lem} \label{lem:coh}
Let $\B$ be a finitely accessible additive category and $\A =
\fpB$. Then the following are equivalent:
\begin{enumerate}
\item $\B$ has products.
\item $\FlatA$ is closed under products in $\ModA$.
\item $\A$ has weak cokernels.
\end{enumerate}
\end{lem}

\begin{proof}
See~\cite[\S 2.1]{CB}.
\end{proof}

\begin{rem}
If $\B$ has products, one can give a more classical proof for
Proposition~\ref{prop:pss_fin_acc}. Namely, one can then replace the
argument by Guil Asensio, Izurdiaga and Torrecillas~\cite{GIT}
by an older and simpler argument by Chase~\cite[Theorem 3.1]{Ch}.
\end{rem}

\section{When is the homotopy category well generated?}
\label{sec:well-gen}

In this final section, we have developed enough tools to answer the
question when exactly is the homotopy category of complexes $\Htp(\B)$
well generated if $\B$ is a finitely accessible additive category. This
way, we will generalize Proposition~\ref{prop:well_gen_rings} and also
give a rather complete answer to \cite[Question 4.2]{HJ} asked by
Holm and J\o{}rgensen. Finally, we will give another criterion
for a triangulated category to be (or not to be) well generated and
this way construct other classes of examples of categories which are
not well generated.

First, we recall a crucial result due to Neeman:

\begin{lem} \label{lem:ProjA}
Let $\A$ be a skeletally small additive category. Then
the homotopy category $\Htp(\ProjA)$ is
$\aleph_1$--well generated. If, moreover, $\A$ has weak cokernels,
then $\Htp(\ProjA)$ is compactly generated.
\end{lem}

\begin{proof}
Neeman has proved in~\cite[Theorem 1.1]{N2} that, given a ring $R$,
the category $\Htp(\ProjR)$ is $\aleph_1$--well generated, and if $R$ is
left coherent then $\Htp(\ProjR)$ is even compactly generated. The
actual arguments, contained in~\cite[\S\S 4--7]{N2}, immediately
generalize to the setting of projective modules over small categories.
The role of finitely generated free modules over $R$ is taken by
representable functors, and instead of the duality between the
categories of left and right projective finitely
generated modules we consider the duality between the idempotent
completions of the categories of covariant and contravariant
representable functors.
\end{proof}

We already know that $\Htp(\B)$ is always locally well generated. When
employing Lemma~\ref{lem:ProjA}, we can show the
following statement, which is one of the main results of this paper:

\begin{thm} \label{thm:fin_acc_well_gen}
Let $\B$ be a finitely accessible additive category. Then the
following are equivalent:
\begin{enumerate}
\item $\Htp(\B)$ is well generated;
\item $\Htp(\B)$ is $\aleph_1$--well generated;
\item $\B$ is pure semisimple.
\end{enumerate}
If, moreover, $\B$ has products, then the conditions are further
equivalent to
\begin{enumerate}
\item[(4)] $\Htp(\B)$ is compactly generated.
\end{enumerate}
\end{thm}

\begin{proof}
$(1) \implies (3)$. If $\Htp(\B)$ is well generated, it is in
particular generated by a set of objects as a localizing subcategory
of itself; see Proposition~\ref{prop:brown}. Hence $\B$ is
pure semisimple by Theorem~\ref{thm:gen_by_set}.

$(3) \implies (2)$ and $(4)$. If $\B$ is pure semisimple and $\A = \fpB$, then
$\B$ is equivalent to $\FlatA$ by Proposition~\ref{prop:repres}, and
$\FlatA = \ProjA$ by Proposition~\ref{prop:pss_fin_acc}. The conclusion
follows by Lemmas~\ref{lem:ProjA} and~\ref{lem:coh}.

$(2)$ or $(4) \implies (1)$. This is obvious.
\end{proof}

\begin{rem}
$(1)$ Neeman proved in~\cite{N2} more than stated in
Lemma~\ref{lem:ProjA}. He described a particular set of generators
for $\Htp(\ProjA)$ satisfying conditions of
Definition~\ref{defn:well_gen}. Namely, $\Htp(\ProjA)$ is always
$\aleph_1$--well generated by a representative set of bounded below
complexes of finitely generated projectives. Moreover, he gave an explicit
description of compact objects in $\Htp(\ProjA)$ in~\cite[7.12]{N2}.

$(2)$ An exact characterization of when $\Htp(\B)$ is compactly generated
and thereby a complete answer to~\cite[Question 4.2]{HJ} does not seem
to be known. We have shown that this reduces to the problem when
$\Htp(\ProjA)$ is compactly generated. A sufficient condition is given
in Lemma~\ref{lem:ProjA}, but it is probably not necessary. On the
other hand, if $R = k[x_1, x_2, x_3, \dots] / (x_i x_j;\; i,j \in \N)$
where $k$ is a field, then $\Htp(\FlatR)$ coincides with
$\Htp(\ProjR)$, but the latter is not a compactly generated
triangulated category; see~\cite[7.16]{N2} for details.
\end{rem}

\begin{expl}
The above theorem adds other locally well generated but not well
generated triangulated categories to our repertoire. For example
$\Htp(\mathcal{TF})$, where $\mathcal{TF}$ stands for the category of
all torsion-free abelian groups, has this property.
\end{expl}

We finish the paper with some examples of triangulated categories
where the fact that they are not generated by a set is less
obvious. For this purpose, we will use the following criterion: 

\begin{prop} \label{prop:well_gen_crit}
Let $\T$ be a locally well generated triangulated category and $\clL$
be a localizing subcategory. Consider the diagram
$$
\begin{CD}
\clL @>{\subseteq}>> \T @>{Q}>> \T/\clL.
\end{CD}
$$
If two of the categories $\clL$, $\T$ and $\T/\clL$ are well
generated, so is the third.
\end{prop}

\begin{proof}
If $\clL = \Loc\clS$ and $\T/\clL = \Loc\C$ for some sets $\clS,\C$,
let $\clL'$ be the localizing subcategory of $\T$ generated by
the set of objects $\clS \cup \C$. Lemma~\ref{lem:loc} yields the
equality $\T/\clL = \clL'/\clL$. Hence also $\T = \clL'$, so $\T$ is
generated by a set, and consequently $\T$ is well generated.

If $\clL$ and $\T$ are well generated, so is $\T/\clL$
by~\cite[7.2.1]{Kr}. Finally, one knows that $X \in \T$ belongs to $\clL$
\iff $QX = 0$; see~\cite[2.1.33 and 1.6.8]{N}. Therefore, if $\T$ and
$\T/\clL$ are well generated, so is $\clL$ by~\cite[7.4.1]{Kr}.
\end{proof}

\begin{rem}
We stress here that by saying that $\T/\clL$ is well generated, we in
particular mean that $\T/\clL$ is a usual category in the sense that
all morphism spaces are sets and \emph{not} proper classes.
\end{rem}

Now we can conclude by showing that some homotopy categories of
acyclic complexes are not well generated.

\begin{expl} \label{expl:acyclic}
Let $R$ be a ring, $\Htp_\textrm{ac}(\ModR)$ be the full subcategory
of $\Htp(\ModR)$ formed by all acyclic complexes, and
$\clL = \Loc\{R\}$. It is well-known but also an easy consequence of
Proposition~\ref{prop:bousfield} that the composition
$$
\Htp_\textrm{ac}(\ModR) \overset{\subseteq}\longrightarrow
\Htp(\ModR) \overset{Q}\longrightarrow \Htp(\ModR) / \clL
$$
is a triangle equivalence between
$\Htp(\ModR)/ \clL$ and $\Htp_\textrm{ac}(\ModR)$.

By Proposition~\ref{prop:well_gen_rings}, $\Htp(\ModR)$ is well
generated \iff $R$ is right pure semisimple. Therefore,
$\Htp_\textrm{ac}(\ModR)$ is well generated \iff $R$ is right
pure semisimple by Proposition~\ref{prop:well_gen_crit}. In fact,
$\Htp_\textrm{ac}(\ModR)$ is not generated by any set of objects if
$R$ is not right pure semisimple. As particular examples, we may take
$R = \Z$ or $R = k(\cdot\!\rightrightarrows\!\cdot)$ for any field $k$.
\end{expl}

\begin{expl} \label{expl:pure_acyclic}
Let $\B$ be a finitely accessible category. Recall that $\B$ is
equivalent to $\FlatA$ for $\A = \fpB$. Then the
natural exact structure on $\FlatA$ coming from $\ModA$ is nothing
else than the well-known exact structure given by pure exact
short sequences in $\B$ (see eg.\ \cite{CB}).

We denote by $\Htp_\textrm{pac}(\FlatA)$ the full subcategory
of $\Htp(\FlatA)$ formed by all
complexes exact with respect to this exact structure, and call such
complexes \emph{pure acyclic}. More explicitly, $\cpx X \in \Htp(\FlatA)$ is
pure acyclic \iff $\cpx X$ is acyclic in $\ModA$ and all the cycles
$Z^i(\cpx X)$ are flat. Note that $\Htp_\textrm{pac}(\FlatA)$ is
closed under taking coproducts in $\Htp(\FlatA)$.

Neeman proved in \cite[Theorem 8.6]{N2} that $\cpx X \in \Htp(\FlatA)$
is pure acyclic \iff there are no non-zero homomorphisms from any
$\cpx Y \in \Htp(\ProjA)$ to $\cpx X$.
Then either by combining Proposition~\ref{prop:bousfield} with
Lemma~\ref{lem:ProjA} or by using \cite[8.1 and 8.2]{N2}, one shows
that the composition
$$
\Htp_\textrm{pac}(\FlatA) \overset{\subseteq}\longrightarrow
\Htp(\FlatA) \overset{Q}\longrightarrow \Htp(\FlatA) / \Htp(\ProjA)
$$
\begin{sloppypar}
\noindent
is a triangle equivalence. Now again,
Proposition~\ref{prop:well_gen_crit} implies that
$\Htp_\textrm{pac}(\FlatA)$ is well generated \iff $\B$ is pure
semisimple. If $\B$ is of the form $\FlatR$ for a ring $R$, this
precisely means that $R$ is right perfect.
\end{sloppypar}

As a particular example, $\Htp_\textrm{pac}(\mathcal{TF})$ is locally
well generated but not well generated, where $\mathcal{TF}$ stands for
the class of all torsion-free abelian groups.
\end{expl}


\end{document}